\documentclass[a4paper, reqno]{amsart}
\usepackage{amsmath, amssymb, eucal, amscd, verbatim, tikz, enumerate}

\newtheorem{thm}{Theorem}[section]
\newtheorem{lem}[thm]{Lemma}
\newtheorem{eg}[thm]{Example}
\newtheorem{prop}[thm]{Proposition}
\newtheorem{cor}[thm]{Corollary}
\newtheorem{defn}[thm]{Definition}
\newtheorem{rem}[thm]{Remark}

\newcommand{\ti}{\tilde}

\newcommand{\bt}{\mathbf{t}}
\newcommand{\br}{\mathbf{r}}
\newcommand{\bs}{\mathbf{s}}
\newcommand{\bp}{\mathbf{p}}
\newcommand{\bq}{\mathbf{q}}
\newcommand{\kk}{\mathfrak{k}}

\newcommand{\BR}{\mathbb{R}}
\newcommand{\RP}{\mathbb{R}_+}
\newcommand{\BN}{\mathbb{N}}
\newcommand{\BC}{\mathbb{C}}

\newcommand{\BZ}{\mathbb{Z}}
\newcommand{\BF}{\mathbb{F}}
\newcommand{\BA}{\mathbb{A}}
\newcommand{\BQ}{\mathbb{Q}}

\newcommand{\CF}{\mathcal{F}}

\newcommand{\KI}{\mathfrak{I}}

\newcommand{\KJ}{\mathfrak{J}}

\newcommand{\BB}{\mathbb{U}}

\newcommand{\KM}{\mathcal{M}}

\begin{document}

\title{The topology on Berkovich affine lines over complete valuation rings}

\author{Chi-Wai Leung and Chi-Keung Ng}

\address[Chi-Wai Leung]{Department of Mathematics, The Chinese University of Hong Kong, Hong Kong.}
\email{cwleung@math.cuhk.edu.hk}

\address[Chi-Keung Ng]{Chern Institute of Mathematics and LPMC, Nankai University, Tianjin 300071, China.}
\email{ckng@nankai.edu.cn}



\date{\today}

\begin{abstract}
In this article, we give a full description of the topology of the one dimensional affine analytic space $\BA_R^1$ over a complete valuation ring $R$ (i.e. a valuation ring with ``real valued valuation'' which is complete under the induced metric), when its field of fractions $K$ is algebraically closed.
In particular, we show that $\BA_R^1$ is both connected and locally path connected. 
Furthermore, $\BA_R^1$ is the completion of $K\times (1,\infty)$ under a canonical uniform structure. %
As an application, we describe the Berkovich spectrum $\KM(\BZ_p[G])$ of the Banach group ring $\BZ_p[G]$ of a cyclic $p$-group $G$ over the ring $\BZ_p$ of $p$-adic integers.
\end{abstract}

\maketitle

\section{Introduction and notation}

Let $S$ be  a commutative unital Banach ring with its norm being denoted by $\|\cdot\|$.
The \emph{Berkovich spectrum} $\KM(S)$ as well as the \emph{$n$-dimensional affine analytic space} $\BA_S^n$ over $S$ was introduced by Vladimir Berkovich (see, e.g., \cite{Berk90}) %
and was further studied by J\'{e}r\^{o}me Poineau in \cite{Poin}.
More precisely, $\KM(S)$ is the set of all non-zero contractive multiplicative semi-norms on $S$, while $\BA_S^n$ is the set of all non-zero multiplicative seminorms on the $n$-variables polynomial ring $S[\bt_{1},...,\bt_{n}]$ whose restrictions on $S$ are contractive.
The topologies on both $\KM(S)$ and $\BA_S^n$ are the ones given by pointwise convergence.
The topology on $\BA_S^n$ is also induced by the \emph{Berkovich uniform structure} which is given by a fundamental system of entourages consisting of sets of the form
\begin{equation}\label{eqt:def-entour}
E_\epsilon^X:=\big\{(\mu, \nu)\in \BA_S^n\times \BA_S^n: \big| |\bp|_{\mu} - |\bp|_{\nu} \big| < \epsilon, \text{ for any }\bp\in X\big\},
\end{equation}
where $\epsilon$ runs through all strictly positive real numbers and $X$ runs through all non-empty finite subsets of $S[\bt_1,...,\bt_n]$.
It is not hard to see that $\BA_S^n$ is complete under this uniform structure.

In the case of a non-Archimedean field $L$, the properties of $\BA_L^n$  plays an important role  in the study of non-Archimedean geometry.
For example, the Bruhat-Tits tree of $SL_{2}({\mathbb Q}_{p})$ can be realized as a subspace of the Berkovich projective line, which is a glue of two copies of $\BA^{1}_{{\mathbb Q}_{p}}$ (see \cite{W}).

When $L$ is an algebraically closed non-Archimedean complete valued field, Berkovich gave in \cite{Berk90} a full
description of the space $\BA_L^1$.
This may then be used to describe the  one-dimensional affine analytic spaces over not necessarily algebraically closed fields.

The aim of this article is to give a full description of the topology space $\BA_R^1$ of a ``complete valuation ring'' $R$.
Recall that an integral domain $R$ is a \emph{valuation ring} if for every element $x$ in its field of fractions $K$, either $x\in R$ or $x^{-1}\in R$ (see e.g. \cite[p.65]{AM} or Definition 2 of \cite[\S VI.1]{Bour}). 
It is easy to see that if $R^\times$ and $K^\times$ are the sets of invertible elements in $R$ and $K$ respectively, then $K^\times / R^\times$ is a totally ordered abelian group and the canonical map $\nu_R: K^\times \to K^\times / R^\times$ is a ``valuation'' such that $R=\{x\in K^\times: \nu_R(x) \geq 0\}\cup \{0_R\}$. 

\begin{defn}
	A valuation ring $R$ is called a \emph{complete valuation ring} if $K^\times / R^\times$ is isomorphic to an ordered subgroup of $\BR$ and $R$ is complete under the induced norm. 
\end{defn}

If $R$ is a complete valuation ring, then $K$ is a non-Archimedean complete valued field 
and $R$ coincides with the ring of integers, $\{a\in K: |a| \leq 1\}$, of $K$.
Conversely, the ring of integers of a non-Archimedean complete valued field is always a complete valuation ring.

Throughout this article, %
for a commutative unital Banach ring $S$, we denote
\begin{equation*}\label{eqt:defn-P-lambda}
S^\star:= S\setminus \{0_S\},
\end{equation*}
where $0_S$ is the zero element of $S$. 
The identity of $S$ will be denoted by $1_S$. 
For any $n\in \BN$,  we define
$$S\{n^{-1}\bt\}:=\left\{\sum_{k=0}^\infty a_k \bt^k: \lim_k \|a_k\|n^k = 0\right\}$$
and equip it with the norm $\left\|\sum_{k=0}^\infty a_k \bt^k\right\| := \max_k \|a_k\|n^k$. %
It is well-known that $S\{n^{-1}\bt\}$ is a commutative unital Banach ring. 
For simplicity, we will use the notation $S\{\bt\}$ for $S\{1^{-1}\bt\}$.

It was shown in \cite{Berk90} that for each $n\in \BN$, the compact space $\KM(S\{n^{-1}\bt\})$ can be identified with the subspace $\{\mu\in \BA_S^1: |\bt|_\mu\leq n\}$ of $\BA_S^1$. 
If we set
\begin{equation}\label{eqt:def-open-n-ball}
\BB^S_n:= \{\mu\in \BA_S^1: |\bt|_\mu < n\}, 
\end{equation}
then $\BA_S^1= \bigcup_{n\in \BN} \BB^S_n = \bigcup_{n\in \BN} \KM(S\{n^{-1}\bt\})$ and hence $\BA_S^1$ is both locally compact and $\sigma$-compact.

From now on, $R$ is a complete valuation ring and $K$ is its field of fractions. 
The absolute value on $K$ induced by $\nu_R$ will be denoted by $|\cdot|$. 
The residue field of $R$ is denoted by $F$ and $Q:R \to F$ is the quotient map.
We set
\begin{equation}\label{eqt:defn-ti-Q}
\ti Q: R[\bt]\to F[\bt]
\end{equation}
to be the map induced by $Q$.

Suppose that $s\in K$ and $\bp\in K[\bt]$.
If $r_0,\dots,r_n\in K$ are the unique elements with $\bp = \sum_{k=0} ^n r_k \bt^k$, then we put $\bp_s:= \sum_{k=0} ^n r_k (\bt+s)^k$.
For any $\lambda\in \BA_K^1$, we define
\begin{equation}\label{eqt:def-lambda+s}
|\bp|_{\lambda + s} := |\bp_s|_\lambda \quad (\bp\in K[\bt]) \qquad \text{and} \qquad \lambda - s:= \lambda + (-s).
\end{equation}
It is easy to see that $\lambda + s\in \BA_K^1$, and that $\lambda\mapsto \lambda + s$ is a bicontinuous bijection from $\BA_K^1$ to itself. 

For every $s\in K$ and $\tau\in \RP$, we denote the closed ball with center $s$ and radius $\tau$ by $D(s,\tau)$, i.e.
$$D(s,\tau):= \{t\in K: |t-s| \leq\tau \},$$
and define (as in \cite{BR})  $\zeta_{s,\tau}\in \BA_K^1$ by $|\bp|_{\zeta_{s,\tau}}:= \sup_{t\in D(s,\tau)} |\bp(t)|$ ($\bp\in K[\bt]$).
Because of the maximum modulus principle, one has
\begin{equation}\label{eqt:def-zeta-s-tau}
\left|\sum_{k=0}^n a_k (\bt - s)^k\right|_{\zeta_{s,\tau}} = \max_{k=0, \dots,n} |a_k| \tau^k \qquad (n\in \BN; a_0,\dots,a_n\in K);
\end{equation}
here, we use the convention that $0^0 := 1$.

Let us recall the following result from \cite{Berk90}.

\begin{thm}\label{thm:berk-cl}
(Berkovich)
Let $L$ be an algebraically closed non-Archimedean complete valued field with a non-trivial norm $|\cdot|$, and $|\cdot|_{\lambda}: L[\bt]\to \RP$ be a function.
Then $\lambda\in \BA_L^1$ if and only if there is a decreasing sequence $\{D(s_n,\tau_n)\}_{n\in \BN}$ of closed balls in $L$ such that $|\bp|_{\lambda} = \inf_{n\in \BN}|\bp|_{\zeta_{s_n, \tau_n}}$.
\end{thm}

The topology on $\BA_L^1$ was also described in \cite{Berk90}.
Moreover, as noted in \cite{Berk90}, by using the fact that $D(t,\tau) = D(s,\tau)$ whenever $D(s,\tau) \subseteq D(t,\tau)$, one can show easily that $\BA_L^1$ is path connected. 

It is also well-known that $\BA_L^1$ is \emph{locally path connected}, in the sense that for every point in this topological space, there is a local neighborhood basis at that point consisting of open sets that are path connected under the induced topologies.
Indeed, as noted in \cite{BR}, any two points in $\BA_L^1$ are joined by a unique path, and hence $\BA_L^1$ is a $\BR$-tree.
The ``weak topology'' induced by this $\BR$-tree structure coincides with the pointwise convergence topology on $\KM(L\{\bt\})$ (see e.g. \cite[Proposition 1.13]{BR}).
Since the canonical basic neighborhoods of the ``weak topology'' are path connected, we know that $\BB^L_1\subseteq \KM(L\{\bt\})$ (see \eqref{eqt:def-open-n-ball}) is locally path connected.
Furthermore, as $\zeta_{s,0} = \zeta_{0_S,0} + s \in \BB^L_1 +s$ (see \eqref{eqt:def-lambda+s}), the density of the image of $L$ in $\BA_L^1$ implies
$\BA_L^1 = \bigcup_{s\in L} \BB^L_1 + s$, and this gives the local path connectedness of $\BA_L^1$.

Observe that if the complete valuation ring $R$ is actually a field, then the absolute value $|\cdot|$ on $K$ is trivial, and the structure of $\BA_R^1$ is already given in \cite[1.4.4]{Berk90}.
However, because we need a concrete presentation of this space for the general case, we will first have a closer look at this case in Proposition \ref{prop:1-D-aff-anal-sp-tri-val}.
As a sidetrack, we verify the fact that if two fields $\kk_1$ and $\kk_2$ are endowed with the trivial norm, then $\BA_{\kk_1}^1 \cong \BA_{\kk_2}^1$ if and only if the cardinalities of the sets of monic irreducible polynomials over them are the same.

Suppose that $R$ is not a field, or equivalently, the absolute value $|\cdot|$ on $K$ is non-trivial.
Let us pick an arbitrary number $\omega\in [1,\infty)$, and
\begin{quote}
	\emph{set $K^\omega$ to be the field $K$ equipped with the equivalent norm $|\cdot|^\omega$.}
\end{quote}
As $|a|^\omega\leq |a|$ ($a\in R$), we know that every semi-norm $\lambda\in \BA_{K^\omega}^1$ restricts to an element in $\BA_R^1$ and this gives %
a map
$$J_\omega^\BA: \BA_{K^\omega}^1 \to \BA_R^1.$$
The map $J_\omega^\BA$ is injective, because for any $\bp\in K[\bt]$, there exists $a\in R$ with $a \bp \in R[\bt]$.
On the other hand, the surjection $\ti Q$ as in \eqref{eqt:defn-ti-Q} produces an injection
$$Q^\BA: \BA_F^1\to \BA_R^1.$$
It is not hard to see that one actually has $\BA_R^1 = Q^\BA(\BA_F^1)\cup \bigcup_{\omega\in [1,\infty)} J^\BA_\omega(\BA_{K^\omega}^1)$ (Proposition \ref{prop:ele-AAS-val-ring}).
Note, however, that in the case of a general Banach integral domain $S$, elements in $\BA_S^1$ cannot be described in such an easy way; for example, if $S$ is the ring $\BZ$ equipped with the trivial norm, the description of elements in $\BA_S^1$ requires the knowledge of all multiplicative ultrametric norms on $\BQ$, instead of just the trivial norm on $\BQ$ (which is the one induced from $S$). 

The topology on $\BA_R^1$ is more difficult to describe, and we will give a full presentation of it in Theorem \ref{thm:AAS-of-val-ring}, in the case when $K$ is algebraically closed.
Using this description, we obtain in Theorem \ref{thm:connect}(c) that  $\BA_R^1$ is first countable if and only if $F$ is countable and $\BA_K^1$ is first countable.
We will also verify, in Proposition \ref{prop:2nd-count}, that $\BA_R^1$ is second countable if and only if $R$ is separable as a metric space  (or equivalently, $K$ is a separable metric space).
Moreover, the Berkovich uniform space $\BA_R^1$ is the completion of $K\times (1,\infty)$ under the induced uniform structure (see Remark \ref{rem:conv}(c)).

We also show that $\BA_R^1$ is both connected and locally path connected (parts (a) and (b) of Theorem \ref{thm:connect}).
Notice that, unlike the case of $\BA_K^1$, any two points in a connected open subset of $\BA_R^1$ are joined by infinitely many paths inside that subset (see Remark \ref{rem:open-connected-subset}).
Consequently, the topology on $\BA_R^1$ cannot be described using the ``weak topology'' of a $\BR$-tree structure.

Finally, we will apply our main result to give a description of the Berkovich spectrum of the Banach group ring $R[G]$ of a cyclic group $G$ over $R$ (Corollary \ref{cor:spec-cyclic-gp}).
In the case when $K$ is not necessarily algebraically closed, one may obtain information about $\KM(R[G])$ by looking at the corresponding spectrum over the completion of the algebraic closure of $K$.
In particular, we will take a closer look at the case when $R=\BZ_p$ and $G$ is a cyclic $p$-group, for a fixed prime number $p$ (Example \ref{eg:order=p}).

\section{The main results}

Let us begin with a careful presentation of the content of the second line of  \cite[1.4.4]{Berk90}.
More precisely, we will give a concrete description of $\BA_\kk^1$ when $\kk$ is a field (not necessarily algebraically closed) equipped with the trivial norm.

In the following, $\kk[\bt]_{\rm irr}$ is the set of all monic irreducible polynomials in $\kk[\bt]$.
Consider $\bq, \bq'\in \kk[\bt]_{\rm irr}$ as well as $\kappa, \kappa'\in \RP$.
We define a semi-norm $\gamma_{\bq, \kappa}$ on $\kk[\bt]$ by
\begin{equation}\label{eqt:defn-lambda-bq-kappa}
\left|\sum_{l=0}^n \br_l \bq^l\right|_{\gamma_{\bq, \kappa}} := \max_{\br_l\neq 0} \kappa^l
\end{equation}
(again, $0^0:=1$),
where $\br_0,\dots,\br_{n-1}\in \kk[\bt]$ and $\br_n\in \kk[\bt]^\star$ are elements with degrees strictly less than $\deg \bq$.
Note that, because the absolute value on $\kk$ is trivial, one has (see \eqref{eqt:def-zeta-s-tau})
\begin{equation}\label{eqt:gamma<->zeta}
\gamma_{\bt - x, \kappa} = \zeta_{x, \kappa} \qquad (x\in \kk;\kappa\in \RP). 
\end{equation}

The semi-norm $\gamma_{\bq,1}$ is independent of $\bq$ and equals the trivial norm on $\kk[\bt]$.
Furthermore, when $\gamma_{\bq, \kappa} = \gamma_{\bq', \kappa'}$, we have $\kappa = \kappa'$, and we will also have $\bq = \bq'$ if, in addition, $\kappa = \kappa'< 1$.
For $\kappa\in (1, \infty)$ and $x\in \kk$, one has $\gamma_{\bt-x,\kappa}(\bp) = \gamma_{\bt, \kappa}(\bp) = \kappa^{\deg \bp}$ ($\bp\in \kk[\bt]$).

Notice that $\gamma_{\bt, \kappa}\in \BA_\kk^1$ for any $\kappa\in \RP$, but $\gamma_{\bq,\kappa}$ is not submultiplicative when $\deg \bq > 1$ and $\kappa > 1$. 
Nevertheless, if $\kappa\in [0,1)$, then $\gamma_{\bq, \kappa}\in \BA_\kk^1$ (regardless of the degree of $\bq$).

On the other hand, for any $\tau\in [-1,\infty)$ and $\bq\in \kk[\bt]_{\rm irr}$, we consider $\delta_{\bq,\tau}$ to be the function from $\kk[\bt]_{\rm irr}$ to $[-1,\infty)$ that vanishes outside the point $\bq$ and sends $\bq$ to $\tau$.
Observe that $\delta_{\bq,0}$ is the constant zero function for all $\bq\in \kk[\bt]_{\rm irr}$.

\begin{prop}\label{prop:1-D-aff-anal-sp-tri-val}
Let $\kk$ be a field endowed with the trivial norm.
Then $\BA_\kk^1$ is canonically homeomorphic to the subspace $X:= \big\{\delta_{\bq,\tau}: \bq \in \kk[\bt]_{\rm irr}\setminus\{\bt\} ; \tau\in  [-1,0)\big\} \cup \big\{\delta_{\bt,\tau}: \tau\in [-1, \infty)\big\}$ of the product space $\prod_{\kk[\bt]_{\rm irr}} [-1,\infty)$.
Consequently, $\BA_\kk^1$ is connected and locally path connected.
Moreover, $\BA_\kk^1$ is first countable if and only if $\kk$ is at most countable.
\end{prop}
\begin{proof}
Since the second and the third statements follow easily from the first one, we will only establish the first statement %
(observe that $\kk[\bt]_{\rm irr}$ is countable if and only if $\kk$ is at most countable). 
Let us show that
\begin{equation}\label{eqt:ele-AAS-trivial}
\BA_\kk^1 =  \{\gamma_{\bq,\kappa}: \bq\in \kk[\bt]_{\rm irr}\setminus \{\bt\}; \kappa\in [0,1)\} \cup \{\gamma_{\bt, \tau}: \tau\in \RP\}.
\end{equation}
In fact, consider any $\lambda\in \BA_\kk^1$ and set $\tau := |\bt|_\lambda$.
If $\tau \in \RP \setminus \{1\}$, one may deduce from the Isosceles Triangle Principle that $\lambda = \gamma_{\bt, \tau}$.

Suppose that $\tau = 1$.
Clearly, $|\bp|_{\lambda} \leq |\bp|_{\gamma_{\bt,1}}$ ($\bp\in \kk[\bt]$) and we consider the case when $\lambda\neq \gamma_{\bt,1}$.
Let
$$P^\lambda:=\{\bp\in \kk[\bt]: |\bp|_\lambda < 1\}.$$
As $P^\lambda$  is a non-zero prime ideal of $\kk[\bt]$, we know that $P^\lambda  = \bq\cdot \kk[\bt]$ for a unique element $\bq\in \kk[\bt]_{\rm irr}\setminus \{\bt\}$ (note that $|\bt|_\lambda = \tau = 1$), and we put
$\kappa :=|\bq|_\lambda\in [0,1)$.
For any $n\in \BZ_+$ and $\br_0,\dots,\br_n\in \kk[\bt]$ with $\br_n \neq 0$ and $\deg \br_k < \deg \bq $ ($k=0,\dots,n$),
we know from the Isosceles Triangle Principle that $\left|\sum_{k=0}^n \br_k \bq^k\right|_\lambda = \left|\sum_{k=0}^n \br_k \bq^k\right|_{\gamma_{\bq, \kappa}}$.

Next, we define a map $\Phi: \BA_\kk^1\to X$ by
$\Phi(\gamma_{\bq, \kappa}) := \delta_{\bq, \kappa - 1}$ ($\gamma_{\bq, \kappa}\in \BA_\kk^1$).
Clearly, $\Phi$ is bijective, and is continuous on the two subsets $\{\gamma_{\bq,\kappa}: \bq\in \kk[\bt]_{\rm irr}; \kappa\in [0,1)\}$ and $\{\gamma_{\bt, \tau}: \tau\in [1,\infty)\}$.
As $\Phi$ restricts to a homeomorphism on the compact subset $\KM(\kk\{\bt\})$ of $\BA_\kk^1$, it is not hard to verify that $\Phi$ is actually bicontinuous.
\end{proof}

If $\kk$ is algebraically closed, then we have, by Equalities \eqref{eqt:gamma<->zeta} and \eqref{eqt:ele-AAS-trivial},
\begin{equation}\label{eqt:BA-kk}
\BA_\kk^1 =  \{\zeta_{s,\kappa}: s\in \kk^\star; \kappa\in [0,1)\} \cup \{\zeta_{0_\kk, \tau}: \tau\in \RP\}.
\end{equation}

Now, consider $\kk_1$ and $\kk_2$ to be two (not necessarily algebraically closed) fields equipped with the trivial norms.
If the cardinalities of $\kk_1[\bt]_{\rm irr}$ and $\kk_2[\bt]_{\rm irr}$ are the same, then Proposition \ref{prop:1-D-aff-anal-sp-tri-val} implies that $\BA_{\kk_1}^1$ is homeomorphic to $\BA_{\kk_2}^1$.
Conversely, suppose that we have a homeomorphism $\Psi: \BA_{\kk_1}^1 \to \BA_{\kk_2}^1$.
For any $\bq_1\in \kk_1[\bt]_{\rm irr}$, it is easy to see that $\Psi$ will send $\gamma_{\bq_1, 0}$ to $\gamma_{\bq_2, 0}$ for some $\bq_2\in \kk_2[\bt]_{\rm irr}$ (and vice-versa), because elements of the form $\gamma_{\bq_i, 0}$ are all the free end points of maximal line-segments of $\BA_{\kk_i}^1$ ($i=1,2$).
Hence, we have a bijection between $\kk_1[\bt]_{\rm irr}$ and $\kk_2[\bt]_{\rm irr}$.
These produce part (a) of the following corollary.
The other parts of this corollary follow from part (a) and some well-known facts.

\begin{cor}
Suppose that $\kk_1$ and $\kk_2$ are two fields equipped with the trivial norm.

\smallskip\noindent
(a) $\BA_{\kk_1}^1$ and $\BA_{\kk_2}^1$ are homeomorphic if and only if the cardinality of $\kk_1[\bt]_{\rm irr}$ equals that of $\kk_2[\bt]_{\rm irr}$.

\smallskip\noindent
(b) If both $\kk_1$ and $\kk_2$ are infinite (as sets), then $\BA_{\kk_1}^1$ and $\BA_{\kk_2}^1$ are homeomorphic if and only if $\kk_1$ and $\kk_2$ has the same cardinality.

\smallskip\noindent
(c) If $\check \kk_1$ is the algebraic closure of $\kk_1$ (again, endowed with the trivial norm), then $\BA_{\kk_1}^1$ is homeomorphic to $\BA_{\check \kk}^1$.
\end{cor}

In the case when the complete valuation ring $R$ is a field, then $|\cdot|$ is a trivial norm, and Proposition \ref{prop:1-D-aff-anal-sp-tri-val} gives the full description of $\BA_R^1$.

\begin{quotation}
\emph{From now on, we consider the case when $R$ is not a field; or equivalently, $|\cdot|$ is non-trivial.}
\end{quotation}

Let us start with the following simple fact.

\begin{lem}\label{lem:prime-id-ring-of-int}
(a) Suppose that $S$ is an integral domain.
If $P\subseteq S[\bt]$ is a prime ideal with $I:=P\cap S$ being non-zero, then one can find a prime ideal $J$ (not necessarily non-zero) of $(S/I)[\bt]$ such that $P= \Upsilon^{-1}(J)$, where $\Upsilon: S[\bt]\to (S/I)[\bt]$ is the quotient map.

\smallskip\noindent
(b) Suppose that $P\subseteq R[\bt]$ is a non-zero prime ideal.
Then $P\cap R\neq \{0_R\}$ if and only if $P= \ti Q^{-1}(J)$ for a (not necessarily non-zero) prime ideal $J$ of $F[\bt]$.
\end{lem}

In fact, as $\ker \Upsilon = I[\bt] \subseteq P$, we know that $P = \Upsilon^{-1}(\Upsilon(P))$ and $\Upsilon(P)$ is a prime ideal of $(S/I)[\bt]$.
Furthermore, part (b)  follows from part (a) and the fact that the only non-zero prime ideal of $R$ is its maximal ideal (see e.g. Propositions 6 and 7 of \cite[\S VI.4.5]{Bour}).

\begin{prop}\label{prop:ele-AAS-val-ring}
Let $R$ be a complete valuation ring, which is not a field. 
Under the notations as in the Introduction, one has
$$\BA_R^1 = Q^\BA(\BA_F^1) \cup \bigcup_{\omega\in [1,\infty)} J_\omega^\BA(\BA_{K^\omega}^1).$$
\end{prop}
\begin{proof}
We have already seen that $Q^\BA(\BA_F^1) \cup \bigcup_{\omega\in [1,\infty)} J_\omega^\BA(\BA_{K^\omega}^1)\subseteq \BA_R^1$.
For the other inclusion, let us pick an arbitrary element $\lambda\in \BA_R^1$.

Suppose that $\ker |\cdot|_\lambda \cap R \neq \{0_R\}$.
Then Lemma \ref{lem:prime-id-ring-of-int}(b) gives a unique prime ideal $J\subseteq F[\bt]$ with $\ker |\cdot|_\lambda = \ti Q^{-1}(J)$.
From this, one concludes that $\lambda = \mu \circ \ti Q$ for an element  $\mu \in \BA_F^1$, i.e. $\lambda\in Q^\BA(\BA_F^1)$.

Suppose that $\ker |\cdot|_\lambda \cap R = \{0_R\}$.
Let us extend $\lambda|_R$ to a function $|\cdot|_{\ti \lambda}:K \to \RP$ by setting  $|r|_{\ti\lambda} := |r^{-1}|_\lambda^{-1}$ whenever $r\in K \setminus R$. 
It is not hard to check that $|\cdot|_{\ti\lambda}$ is a multiplicative norm on $K$.
For any $r\in K$, one has $|r| \leq 1$  if and only if $|r|_{\ti\lambda} \leq 1$.
Thus, $|\cdot|_{\ti \lambda}$ is equivalent to $|\cdot|$, and one can find $\omega\in \RP$ satisfying $|\cdot|_{\ti\lambda} = |\cdot|^\omega$ (see e.g. Proposition 3 of \cite[\S VI.3.2]{Bour}).
Moreover, as $|a|_{\ti \lambda}  \leq |a|$ ($a\in R$), we know that $\omega\geq 1$.
Finally, if we put
$$|a^{-1}\bp|_{\bar \lambda} := |a|_\lambda^{-1} |\bp|_\lambda \qquad (a\in R^\star;\bp\in R[\bt]),$$
then $|\cdot|_{\bar \lambda}$ is a well-defined multiplicative seminorm on $K[\bt]$ with $|r|_{\bar \lambda} = |r|_{\ti \lambda}$ ($r\in K$).
Thus, $\bar \lambda\in \BA_{K^\omega}^1$ and $\lambda = J_\omega^\BA(\bar \lambda)$.
\end{proof}

It is clear that both $Q^\BA: \BA_F^1\to \BA_R^1$ and $J_\omega^\BA:\BA_{K^\omega}^1 \to \BA_R^1$ ($\omega\in [1,\infty)$) are homeomorphisms onto their images  (by considering the compact spaces $\KM(F\{ n^{-1}\bt\})$ and $\KM(K^\omega\{ n^{-1}\bt\})$ for all $n\in \BN$).
If $\lambda\in \BA_K^1$ and $\omega\in [1,\infty)$, let us define $\lambda^\omega\in \BA_{K^\omega}^1$ by
$$|\bp|_{\lambda^\omega} := |\bp|_\lambda^\omega \qquad (\bp\in K[\bt]).$$
Obviously, $(\lambda,\omega) \mapsto J_\omega^\BA(\lambda^\omega)$ induces a continuous bijection
\begin{equation}\label{eqt:def-Lambda}
\Lambda: \BA_K^1\times [1,\infty) \to \bigcup_{\omega\in [1,\infty)} J_\omega^\BA(\BA_{K^\omega}^1).
\end{equation}
Later on, we will verify that $\Lambda$ is actually a homeomorphism.
We will also describe how the subspaces $Q^\BA(\BA_{K^\omega}^1)$ and $\bigcup_{\omega\in [1,\infty)} J_\omega^\BA(\BA_{K^\omega}^1)$ sit together in $\BA_R^1$.

\begin{quotation}
\emph{From now on, we will identify $\BA_F^1$ as subspaces of $\BA_R^1$, and may sometimes ignore the map $Q^\BA$ if no confusion arises.}
\end{quotation}

\begin{lem}\label{lem:top-type-I-norm<1}
Suppose that $F$ is algebraically closed.
Consider a net $\{(\lambda_i, \omega_i)\}_{i\in \KI}$ in $\BA_K^1\times [1, \infty)$.
	
	\smallskip\noindent
	(a) If $\{\lambda_{i}^{\omega_i}\}_{i\in \KI}$ converges to an element in  $\BA_F^1$, then $\lim_i \omega_i = \infty$ and $\limsup_{i\in \KI} |\bt|_{\lambda_i}\leq 1$.
	
	\smallskip\noindent
	(b) If $\lambda_i\in \BB^K_1$ for all $i\in \KI$ (see \eqref{eqt:def-open-n-ball}), $\lim_i \omega_i = \infty$ and $\lim_i |\bt|_{\lambda_i}^{\omega_i} = \tau\in [0,1]$, then $\lambda_{i}^{\omega_i} \to \zeta_{0_F, \tau}$.
\end{lem}
\begin{proof}
	(a) By \eqref{eqt:BA-kk} and the assumption, one can find $(x, \tau)\in F^\star \times [0,1) \cup \{0_F\}\times \RP$ such that $\lambda_{i}^{\omega_i} \to \zeta_{x , \tau}$.
	Assume on the contrary that $\{\omega_i\}_{i\in \KI}$ has a bounded subnet.
	Then there is a subnet $\{\omega_{i_j}\}_{j\in J}$ of $\{\omega_i\}_{i\in \KI}$ converging to a number $\omega_0\in [1,\infty)$.
	For any $a\in R$ with $|a| <1$, one gets from 
	$$|a|^{\omega_{i_j}} = |a|_{\lambda_{i_j}^{\omega_{i_j}}}\to |Q(a)|_{\zeta_{x , \tau}} = 0,$$
	that $|a|^{\omega_0} = 0$, and this contradicts the %
	non-trivial assumption of $|\cdot|$.
	Hence, $\omega_i \to \infty$.
	On the other hand, since $\{|\bt|_{\lambda_i}^{\omega_i}\}_{i\in \KI}$ converges to $\big|\ti Q(\bt)\big|_{\zeta_{x , \tau}}$, one concludes that $\limsup_{i\in \KI} |\bt|_{\lambda_i} \leq 1$, because otherwise, there exist $r > 1$ and a subnet $\{\lambda_{i_j}\}_{j\in \KJ}$ with $|\bt|_{\lambda_{i_j}} \geq r$, which produces the contradiction that $|\bt|_{\lambda_{i_j}}^{\omega_{i_j}}\geq r^{\omega_{i_j}}\to \infty$.
	
	\smallskip\noindent
	(b) If $b\in R$ with $|b| = 1$, then $|\bt - b|_{\lambda_{i}^{\omega_i}} = 1$ (because $|\bt|_{\lambda_i}< 1$) for all $i\in \KI$.
	On the other hand, consider a polynomial $\bq\in \ker \ti Q$.
	There exist $a_0,\dots,a_n\in \ker Q$ with $\bq = \sum_{k=0}^n a_k \bt^k$.
	Hence,
	$$|\bq|_{\lambda_i} \leq \max_{k=0,\dots,n} |a_k||\bt|_{\lambda_i}^k \leq \max_{k=0,\dots,n} |a_k| < 1,$$
	and one has $|\bq|_{\lambda_{i}^{\omega_i}} \leq \max_{k=0,\dots,n} |a_k|^{\omega_i} \to 0$ (along $i$).
	
	Now, let $\bp\in R[\bt]^\star$ and $\{x_1,\dots,x_m\}$ be all the non-zero roots of $\ti Q(\bp)$ in $F$ (counting multiplicity).
	Pick $b_1,\dots,b_m\in R$ with $Q(b_l) = x_l$ ($l=1,\dots,m$).
	Then $|b_l| = 1$ for all $l\in\{1,\dots,m\}$, and
	one can find $k\in \BZ_+$ as well as $\bq_0\in \ker \ti Q$ satisfying
	$$\bp = \bt^k \cdot (\bt - b_1)\cdots (\bt - b_m) + \bq_0. $$
	The above and the hypothesis will then tell us that $|\bt^k \cdot (\bt - b_1)\cdots (\bt - b_m)|_{\lambda_{i}^{\omega_i}} \to \tau^k$ and $|\bq_0|_{\lambda_{i}^{\omega_i}} \to 0$.
	Consequently, $|\bp|_{\lambda_{i}^{\omega_i}} \to \tau^k = |\ti Q(\bp)|_{\zeta_{0_F, \tau}}$ as is required (observe that $0\leq \tau \leq 1$).
\end{proof}

The following is our first main theorem that gives a full description of the topological space $\BA_R^1$.

\begin{thm}\label{thm:AAS-of-val-ring}
Let $R$ be a complete valuation ring which is not a field, $F$ be the residue field of $R$, and $K$ be the field of fractions of $R$ equipped with the induced absolute value $|\cdot|$.
Suppose that $K$ is algebraically closed. 
We 
fix a cross section $\ti F\subseteq R$ of $Q:R\to F$ that contains $0_R$.

\smallskip\noindent
(a) $\BA_F^1$ is closed in $\BA_{R}^1$.

\smallskip\noindent
(b) The map $\Lambda: \BA_K^1\times [1,\infty) \to \bigcup_{\omega\in [1,\infty)} J_\omega^\BA(\BA_{K^\omega}^1)$ 
in \eqref{eqt:def-Lambda} is a homeomorphism.

\smallskip\noindent
(c) Suppose that $\{(\lambda_i, \omega_i)\}_{i\in \KI}$ is a net in $\BA_K^1\times [1,\infty)$.
Then $\{\lambda_{i}^{\omega_i}\}_{i\in \KI}$ converges to an element $\lambda_0\in \BA_F^1$ if and only if $\omega_i \to \infty$ and either one of the following holds:
\begin{enumerate}[\ \ C1).]
	\item there exist $\tau_1\in [0,1)$ and $b\in R$ such that $|\bt - b|_{\lambda_i}^{\omega_i} \to \tau_1$ (in this case, $\lambda_0= \zeta_{Q(b), \tau_1}$);
	
	\item one can find $\tau_2\in (1,\infty)$ such that $|\bt|_{\lambda_i}^{\omega_i}\to \tau_2$ (in this case, $\lambda_0 = \zeta_{0_F, \tau_2}$);

	\item $|\bt - c|_{\lambda_i}^{\omega_i}\to 1$ for any $c\in \ti F$ (in this case, $\lambda_0 = \zeta_{0_F,1}$).
\end{enumerate}

\smallskip\noindent
(d) Under the homeomorphism in part (b), one may regard  $\BA_{K}^1\times [1,\infty)$  as an open dense subset of $\BA_R^1$.
Consequently, the image of $K\times (1,\infty)$ is dense in $\BA_R^1$.
\end{thm}
\begin{proof}
The non-triviality of $|\cdot|$ produces $a_0\in R$ with $|a_0|\in (0,1)$.
Moreover, since $K$ is algebraically closed, so is $F$.
As in \eqref{eqt:BA-kk}, one has
$$\BA_F^1 = \big\{\zeta_{x, \tau}: (x,\tau)\in F\times [0,1)\cup \{0_F\}\times [1,\infty)\big\}.$$

\smallskip\noindent
(a) Suppose on the contrary that there is a net $\{(x_i, \tau_i)\}_{i\in \KI}$ in $F\times [0,1) \cup \{0_F\}\times [1,\infty)$ with $\zeta_{x_i, \tau_i}\to \lambda^{\omega}$ for some $(\lambda, \omega)\in \BA_K^1\times [1,\infty)$.
	Then, $0  = |Q(a_0)|_{\zeta_{x_i, \tau_i}} \to |a_0|^{\omega}$, which is absurd.
	
	\smallskip\noindent
	(b) As said in the paragraph preceding Lemma \ref{lem:top-type-I-norm<1}, the map $\Lambda$ is a continuous bijection.
	If $\{(\lambda_i, \omega_i)\}_{i\in \KI}$ is a net in $\BA_K^1\times [1, \infty)$ satisfying $\lambda_i^{\omega_i} \to \lambda_0^{\omega_0}$ for some $(\lambda_0, \omega_0)\in \BA_K^1\times [1, \infty)$, then  $$|a_0|^{\omega_i/\omega_0} = |a_0|_{\lambda_i^{\omega_i}}^{1/\omega_0}\to |a_0|$$ and so $\omega_i \to \omega_0$.
	From this, one can also deduce that $|\bp|_{\lambda_i} \to |\bp|_{\lambda_0}$ for every $\bp\in R[\bt]$.
Consequently, the inverse of $\Lambda$ is continuous.
	
	\smallskip\noindent
	(c) $\Rightarrow)$.
Suppose that $\{\lambda_{i}^{\omega_i}\}_{i\in \KI}$ converges to $\zeta_{Q(b), \tau_1}$ for some $b\in R$ and $\tau_1\in [0,1)$.
	Then we have $\omega_i \to \infty$ (by Lemma \ref{lem:top-type-I-norm<1}(a)) and $|\bt -b|_{\lambda_i}^{\omega_i} \to |\bt - Q(b)|_{\zeta_{Q(b), \tau_1}} = \tau_1$.
	This verifies Condition (C1).
	
	On the other hand, suppose that $\{\lambda_{i}^{\omega_i}\}_{i\in \KI}$ converges to $\zeta_{0_F, \kappa}$ for some $\kappa\in [1,\infty)$.
	Again, one has $\omega_i\to \infty$ because of Lemma \ref{lem:top-type-I-norm<1}(a).
	Moreover, as $\kappa\geq 1$, one knows that
	$$|\bt - c|_{\lambda_i}^{\omega_i} \to \left|\bt - Q(c)\right|_{\zeta_{0_F, \kappa}} = \kappa \qquad (c\in R).$$
Hence, either Conditions (C2) or (C3) holds (depending on whether $\kappa =1$).

\smallskip\noindent
$\Leftarrow)$.
Suppose that $\omega_i \to \infty$ and Condition (C1) holds.
	Then $|\bt|_{\lambda_i -b}^{\omega_i} = |\bt - b|_{\lambda_i}^{\omega_i} \to \tau_1 < 1$ (see \eqref{eqt:def-lambda+s}), which implies that $|\bt|_{\lambda_i -b} < 1$ eventually.
By Lemma \ref{lem:top-type-I-norm<1}(b), we know that $(\lambda_i -b)^{\omega_i} \to \zeta_{0_F, \tau_1}$ and hence $\lambda_i^{\omega_i} \to \zeta_{Q(b), \tau_1}$.
	
Secondly, suppose that $\omega_i \to \infty$ and Condition (C2) holds.
We first note that
	\begin{equation}\label{eqt:limsup-leq-1}
	{\limsup}_{i\in \KI} |\bt|_{\lambda_i} \leq 1,
	\end{equation}
	because otherwise, one can find a subnet such that  $|\bt|_{\lambda_{i_j}}^{\omega_{i_j}}\to \infty$.
Let us also show that 
\begin{equation}\label{eqt:all-elem-R}
|\bt - c'|_{\lambda_i}^{\omega_i}\to \tau_2 \qquad (c'\in R).
\end{equation}
In fact, as $\tau_2 > 1$, when $i$ is large, one has $|\bt|_{\lambda_i} > 1$, and hence $|\bt-c'|_{\lambda_i} = |\bt|_{\lambda_i}$ ($c'\in R$) and Relation \eqref{eqt:all-elem-R} follows.

Now, for any $\bq\in \ker \ti Q$, we let $a_0,\dots,a_n\in \ker Q$ be the elements with $\bq = \sum_{k=0}^n a_k \bt^k$.
As $\kappa:=\max\{|a_0|,\dots,|a_n|\} < 1$, we obtain from  \eqref{eqt:limsup-leq-1} an element $i_0\in \KI$ satisfying
	$\sup_{i\geq i_0}|\bt|_{\lambda_i} < 1/\kappa^{\frac{1}{n+1}}$ and hence for $i\geq i_0$,
$$|\bq|_{\lambda_i} \leq \max_{k=0,\dots,n} |a_k||\bt|_{\lambda_i}^k  < \kappa^{\frac{1}{n+1}}.$$
This means that $|\bq|_{\lambda_i^{\omega_i}} \to 0$.
Consider now a polynomial $\bp\in R[\bt]$ of degree $m$.
Let $\{y_1,\dots, y_m\}$ be all the roots of $\ti Q(\bp)$ in $F$ (counting multiplicity) and let $c_1,\dots,c_m$ be elements in $R$ satisfying $Q(c_l) = y_l$ ($l=1,\dots,m$).
Then one can find $\bq_0\in \ker \ti Q$ with
$$\bp = (\bt - c_1)\cdots (\bt - c_m) + \bq_0.$$
Thus, \eqref{eqt:all-elem-R} and the above imply that $|\bp|_{{\lambda_i^{\omega_i}}} \to \tau_2^m = \zeta_{0_F,\tau_2}(\ti Q(\bp))$ (notice that $\tau_2 \geq 1$).

Finally, we consider the case when $\omega_i \to \infty$ and Condition (C3) holds.
As $0_R\in \ti F$, we know that $|\bt|_{\lambda_i}^{\omega_i}\to 1$ and Relation \eqref{eqt:limsup-leq-1} is satisfied.
Moreover, we have $|\bt - c'|_{\lambda_i}^{\omega_i}\to 1$ $(c'\in R)$.
Indeed, if $c'\in R$, one can find $c\in \ti F$ satisfying $|c'-c|<1$.
As $\omega_i \to \infty$, there exists $i_0\in \KI$ such that for any $i\geq i_0$, one has both $|c'-c|^{\omega_i} < 1/2$ as well as
$|\bt - c|_{\lambda_i}^{\omega_i}> 1/2$, and hence
$$|\bt - c'|_{\lambda_i} = |\bt - c|_{\lambda_i} \qquad (i\geq i_0).$$
This implies that $|\bt - c'|_{\lambda_i}^{\omega_i} \to 1$ as required.
Now, it follows from the same argument as that for Condition (C2) that $|\bp|_{{\lambda_i^{\omega_i}}} \to 1 = \zeta_{0_F,1}(\ti Q(\bp))$, for any $\bp\in R[\bt]$.

\smallskip\noindent
(d) If $\tau_1 < 1$, then it follows easily from part (c) that $\zeta_{b, \tau_1^{1/n}}^n \to \zeta_{Q(b), \tau_1}$ for every $b\in R$.
If $\tau_2 > 1$, it is not hard to see from part (c) that $\zeta_{0_K, \tau_2^{1/n}}^n \to \zeta_{0_F, \tau_2}$.
Furthermore, since $\zeta_{0_K, 1}(\bt - c) = 1$ for all $c\in R$, we know from part (c) that $\zeta_{0_K, 1}^n \to \zeta_{0_F, 1}$.
These establish the first statement.
The second statement follows from the first one, part (b) as well as the fact that $K$ is dense in $\BA_K^1$.
\end{proof}

\begin{quotation}
\emph{From now on, we will also identify $\BA_K^1\times [1,\infty)$ with $\bigcup_{\omega\in [1,\infty)} J_\omega^\BA(\BA_{K^\omega}^1)$ (i.e. $(\lambda, \omega)$ is identified with $\lambda^\omega$) as well as $K\times [1,\infty)$ with its image in $\bigcup_{\omega\in [1,\infty)} J_\omega^\BA(\BA_{K^\omega}^1)$ (i.e. $(r, \omega)$ is identified with $\zeta_{r,0}^\omega$).} 
\end{quotation}

\begin{rem}\label{rem:conv}
(a) In the proof of Theorem \ref{thm:AAS-of-val-ring}(c), we see that $\lambda_i^{\omega_i} \to \zeta_{0_F,\kappa}$ for some $\kappa\in [1,\infty)$ if and only if
$\omega_i \to \infty$ and $|\bt - c'|_{\lambda_i}^{\omega_i}\to \kappa$, for all $c'\in R$.

\smallskip\noindent
(b) Unlike the case of $\tau_2 \in (1, \infty)$, the requirement $|\bt|_{\lambda_i}^{\omega_i}\to 1$ does not imply
$\lambda_i^{\omega_i} \to \zeta_{0_F,1}$.
For example, if we set $\lambda_i := \zeta_{1_K,0}$ ($i\in \KI$), then $|\bt|_{\lambda_i}^{\omega_i} =  1$ but $|\bt - 1|_{\lambda_i}^{\omega_i}=0$ for all $i\in \KI$.

\smallskip\noindent
(c) Consider the uniform structure on $K\times (1,\infty)$ that is induced by the fundamental system of entourages of the form:
$$\big\{ \big((s_1,\omega_1), (s_2, \omega_2)\big): \big| |\bp(s_1)|^{\omega_1} - |\bp(s_2)|^{\omega_2}\big| < \epsilon, \text{ for each } \bp\in X\big\},$$
where $\epsilon$ runs through all positive numbers and $X$ runs through all non-empty finite subsets of $R[\bt]$.
As a uniform space, $\BA_R^1$ is the completion of $K\times (1,\infty)$ under this structure (because of Theorem \ref{thm:AAS-of-val-ring}(d)).
\end{rem}

Before continuing our discussion, let us set some more notations.
For any $\bp\in R[\bt]$ as well as $\tau, \epsilon\in \RP$, we denote
$$U^\bp_{\tau, \epsilon} := \{\mu\in \BA_R^1: \tau - \epsilon < \mu(\bp) < \tau + \epsilon\}.$$
It follows from the definition that $U^\bp_{\tau, \epsilon}$ an open subset of $\BA_R^1$.

Proposition \ref{prop:ele-AAS-val-ring} as well as parts (a) and (b) of Theorem \ref{thm:AAS-of-val-ring} tell us that $\BA_R^1$ contains the product space $\BA_K^1\times [1,\infty)$ as an open subset with its complement being $\BA_F^1$.
Moreover, 
by Lemma \ref{lem:top-type-I-norm<1}(a), we know that for every $\kappa \geq 1$, the subset $\BA_K^1\times [1,\kappa]$ is closed in $\BA_R^1$.

Since the topology on the product space $\BA_K^1\times [1,\infty)$ is well-known %
and Proposition \ref{prop:1-D-aff-anal-sp-tri-val} describes the topological space $\BA_F^1$, the topology of $\BA_R^1$ can be determined if one knows the description of neighborhood bases over elements in $\BA_F^1$. %
Through Theorem \ref{thm:AAS-of-val-ring}, these bases will be described as follows:

\begin{enumerate}[N1)]
\item Consider $\tau\in [0,1)$ and $b\in \ti F$ (see Theorem \ref{thm:AAS-of-val-ring}).
If $\kappa \geq 1$ and $0 < \epsilon < 1-\tau$, one can check easily, %
using Relation \eqref{eqt:BA-kk}, that
\begin{align*}
& U^{\bt - b}_{\tau ,\epsilon}\setminus \BA_K^1\times [1,\kappa] =  \\
& \  \left\{(\lambda, \omega)\in \BA_K^1\times [1,\infty): \omega > \kappa ; \big| |\bt - b|_\lambda^\omega - \tau\big| < \epsilon\right\}\cup  \big\{\zeta_{Q(b), \upsilon}: \upsilon\in [0,1); |\upsilon - \tau| < \epsilon\big\}.
\end{align*}
Thus, by Theorem \ref{thm:AAS-of-val-ring}(c) and Proposition \ref{prop:1-D-aff-anal-sp-tri-val}, the countable collection:
$$\big\{U^{\bt - b}_{\tau ,1/n}\setminus \BA_K^1\times [1,m]: m,n \in \BN \text{ such that } 1/n \in (0,1 - \tau)\big\}$$
constitutes an open neighborhood basis of $\zeta_{Q(b), \tau}$.

\item Consider $\tau > 1$.
If $\kappa \geq 1$ and $0 < \epsilon < \tau -1$, one may verify that
\begin{align*}
& U^{\bt}_{\tau ,\epsilon}\setminus \BA_K^1\times [1,\kappa] =  \\
& \ \left\{(\lambda, \omega)\in \BA_K^1\times [1,\infty): \omega > \kappa ; \big| |\bt|_\lambda^\omega - \tau\big| < \epsilon\right\}\cup  \big\{\zeta_{0_F, \upsilon}: \tau - \epsilon < \upsilon < \tau + \epsilon\big\}.
\end{align*}
Again, Theorem \ref{thm:AAS-of-val-ring}(c) and Proposition \ref{prop:1-D-aff-anal-sp-tri-val} imply that the countable collection:
$$\big\{U^{\bt}_{\tau ,1/n}\setminus \BA_K^1\times [1,m]: m,n \in \BN \text{ such that } 1/n < \tau - 1 \big\}$$
constitutes an open neighborhood basis of $\zeta_{0_F, \tau}$.

\item Let $\CF$ be the collection of all finite subsets of $\ti F$ such that each of them contains $0_R$.
If $X\in \CF$, $\kappa \geq 1$ and $\epsilon \in (0, 1)$, it is not hard to see that
\begin{align*}
\bigcap_{c\in X} &  U^{\bt - c}_{1 ,\epsilon}\setminus \BA_K^1\times [1,\kappa] = \\
&  \left\{(\lambda, \omega)\in \BA_K^1\times [1,\infty): \omega > \kappa ; \big| |\bt - c|_\lambda^\omega - 1\big| < \epsilon, \text{ for any }c\in X\right\} \cup  \\
& \ \big\{\zeta_{Q(c), \upsilon}: c\in X; \upsilon\in (1-\epsilon,1)\big\}  \cup \big\{\zeta_{0_F, \upsilon}: \upsilon\in [1, 1+\epsilon)\big\} \cup \big\{\zeta_{Q(b), \upsilon}: b\in \ti F\setminus X; \upsilon\in [0,1)\big\} .
\end{align*}
It now follows from Theorem \ref{thm:AAS-of-val-ring}(c) and Proposition \ref{prop:1-D-aff-anal-sp-tri-val} that the collection:
$$\left\{\bigcap_{c\in X} U^{\bt - c}_{1 ,1/n}\setminus \BA_K^1\times [1,m]: m,n \in \BN; X\in \CF \right\}$$
constitutes an open neighborhood basis of $\zeta_{0_F, 1}$.
\end{enumerate}

One may use the above information to obtain certain topological properties of $\BA_R^1$.
The following is an illustration.

\begin{thm}\label{thm:connect}
Let $R$ be a complete valuation ring which is not a field, with its field of fractions being algebraically closed.
The following properties hold. 

\smallskip\noindent
(a) $\BA_R^1$ is path connected.

\smallskip\noindent
(b) $\BA_R^1$ is locally path connected.

\smallskip\noindent
(c) $\BA_R^1$ is first countable if and only if $F$ is countable and $\BA_K^1$ is first countable.
\end{thm}
\begin{proof}
(a) By Proposition \ref{prop:1-D-aff-anal-sp-tri-val}, $\BA_F^1$ is path connected.
Moreover, as said in the Introduction, $\BA_K^1$ is path connected, and this gives the path connectedness of $\BA_K^1\times [1,\infty)$.
Consider $(x_0, \tau_0)\in F\times [0,1)\cup \{0_F\}\times [1,\infty)$ and $(\lambda_0,\omega_0)\in \BA_K^1\times [1,\infty)$.
As in the proof of Theorem \ref{thm:AAS-of-val-ring}(d), one has $\zeta_{0_K,0}^\omega \to \zeta_{0_F,0}$ (when $\omega\to \infty$), and this produces a path joining $\zeta_{0_K,0}$ to  $\zeta_{0_F,0}$.
By considering a path in $\BA_F^1$ (respectively, $\BA_K^1\times [1,\infty)$) joining $\zeta_{x_0,\tau_0}$ to $\zeta_{0_F,0}$ (respectively, joining  $\lambda_0^{\omega_0}$ to $\zeta_{0_K,0}$), one obtains a path that joins $\zeta_{x_0,\tau_0}$ to $\lambda_0^{\omega_0}$.

\smallskip\noindent
(b) As recalled in the Introduction, $\BA_K^1$ is locally path connected, and hence so is the open subset $\BA_K^1\times [1,\infty)$ of $\BA_R^1$.

Let us now verify the path connectedness of open sets of the form as in (N1). 
Let $\tau\in [0,1)$ and $b\in \ti F$.
Fix $\kappa \geq 1$ as well as $\epsilon \in (0, 1-\tau)$. 
For simplicity, we denote %
$$V^\kappa := \BA_K^1\times (\kappa,\infty).$$
It is easy to see that $\big\{\zeta_{Q(b), \upsilon}: \upsilon\in [0,1); |\upsilon - \tau| < \epsilon\big\}$ is path connected (see Proposition \ref{prop:1-D-aff-anal-sp-tri-val}).
Thus, in order to verify the path connectedness of $U^{\bt - b}_{\tau ,\epsilon}\setminus \BA_K^1\times [1,\kappa]$, it suffices to show that an arbitrary fixed element $\lambda_0^{\omega_0} \in U^{\bt - b}_{\tau ,\epsilon}\cap V^\kappa$ can be joined to $\zeta_{Q(b),\tau}$ through a path inside $U^{\bt - b}_{\tau ,\epsilon}\setminus \BA_K^1\times [1,\kappa]$.

Notice that the subset
$$\left\{\lambda\in \BA_K^1: \big| |\bt - b|_\lambda^{\omega_0} - \tau\big| < \epsilon\right\}$$
is open in $\BA_K^1$ and contains $\lambda_0^{\omega_0}$.
Hence, by the locally path connectedness of $\BA_K^1$ and the density of
the image of $K$ in $\BA_K^1$, one may assume that $\lambda_0 = \zeta_{s_0, 0}$ for an element $s_0\in K$.
The requirement of $\zeta_{s_0,0}^{\omega_0}\in U^{\bt - b}_{\tau ,\epsilon}$ implies that
$$\tau - \epsilon < |s_0-b|^{\omega_0} < \tau + \epsilon.$$
For every $\upsilon \in \big[0, |s_0-b|\big]$, the relation
\begin{equation*}\label{eqt:bt-b-s0-upsilon}
|\bt - b|_{\zeta_{s_0, \upsilon}}^{\omega_0}  = \max \{ \upsilon, |s_0-b|\}^{\omega_0} = |s_0-b|^{\omega_0}\in (\tau - \epsilon, \tau+ \epsilon)
\end{equation*}
tells us that $\zeta_{s_0, \upsilon}^{\omega_0}\in U^{\bt - b}_{\tau ,\epsilon}$.
Thus, $\big\{\zeta_{s_0, \upsilon}^{\omega_0}: \upsilon\in [0, |s_0-b|]\big\}$ is a path in $U^{\bt - b}_{\tau ,\epsilon}$ joining $\zeta_{s_0,0}^{\omega_0}$ to $\zeta_{s_0, |s_0-b|}^{\omega_0} = \zeta_{b, |s_0-b|}^{\omega_0}$.

Similarly, the path
$\big\{\zeta_{b, \upsilon}^{\omega_0}: \upsilon \text{ is in between } |s_0-b| \text{ and } \tau^{1/\omega_0}\big\}$
that joins $\zeta_{b, |s_0-b|}^{\omega_0}$ to $\zeta_{b, \tau^{1/\omega_0}}^{\omega_0}$ also lies inside $U^{\bt - b}_{\tau ,\epsilon}$.
Moreover, it follows from
\begin{equation}\label{eqt:path}
|\bt - b|_{\zeta_{b, \tau^{1/\omega}}}^\omega  = \tau
\end{equation}
and Theorem \ref{thm:AAS-of-val-ring}(c) that the net $\{\zeta_{b, \tau^{1/\omega}}^{\omega}\}_{\omega\geq \omega_0}$ converges to $\zeta_{Q(b),\tau}$ when $\omega\to \infty$.
Therefore,
$$\big\{\zeta_{b, \tau^{1/\omega}}^{\omega}: \omega\geq \omega_0\big\}\cup \big\{\zeta_{Q(b),\tau}\big\}$$
is a path in $\BA_K^1$ joining $\zeta_{b, \tau^{1/\omega_0}}^{\omega_0}$ to $\zeta_{Q(b),\tau}$.
Furthermore, \eqref{eqt:path} also ensures that this path lies inside $U^{\bt - b}_{\tau ,\epsilon}$.
Consequently, $U^{\bt - b}_{\tau ,\epsilon}\setminus \BA_K^1\times [1,\kappa]$ is path connected (note that $\omega_0$ as well as all the $\omega$ in the above are strictly bigger than $\kappa$).

In the same way, for any fixed $\tau\in (1,\infty)$, one can establish the path connectedness of open sets as in (N2), i.e., $U^{\bt}_{\tau ,\epsilon}\setminus \BA_K^1\times [1,\kappa]$ for $\kappa \geq 1$ and $\epsilon \in (0, \tau - 1)$.

It remains to show that open sets of the form as in (N3), namely, the set $\bigcap_{c\in X} U^{\bt - c}_{1 ,\epsilon}\setminus \BA_K^1\times [1,\kappa]$ (for $X\in \CF$, $\kappa \geq 1$ and $\epsilon\in (0,1)$) is path connected.
As in the above, we need to find a path in $\bigcap_{c\in X} U^{\bt - c}_{1 ,\epsilon}\setminus \BA_K^1\times [1,\kappa]$ that joins an arbitrary chosen element $\lambda_0^{\omega_0}\in \bigcap_{c\in X} U^{\bt - c}_{1 ,\epsilon}\cap V^\kappa$ to $\zeta_{0_F,1}$.
Again, we may assume that $\lambda_0 = \zeta_{s_0,0}$ for an element $s_0\in K$.
The condition $\zeta_{s_0,0}^{\omega_0}\in \bigcap_{c\in X} U^{\bt - c}_{1 ,\epsilon}$ implies that
\begin{equation*}\label{eqt:s0-c}
(1 - \epsilon)^{1/\omega_0} < |s_0-c| < (1 + \epsilon)^{1/\omega_0} \qquad (c\in X),
\end{equation*}
and, in particular, $|s_0|^{\omega_0}\in (1 - \epsilon, 1 + \epsilon)$ (because $0_R\in X$ by the assumption of $\CF$).
For every $\upsilon\in [0,|s_0|]$, the equality
\begin{equation}\label{eqt:bt-c-s0-upsilon}
|\bt - c|_{\zeta_{s_0, \upsilon}}^{\omega_0}  = \max \{ \upsilon, |s_0-c|\}^{\omega_0} \qquad (c\in X)
\end{equation}
will then ensure that $\zeta_{s_0,\upsilon}^{\omega_0}\in \bigcap_{c\in X} U^{\bt - c}_{1 ,\epsilon}$.
In addition, a similar relation as \eqref{eqt:bt-c-s0-upsilon} tells us that the path $\big\{\zeta_{0_K, \upsilon}^{\omega_0}: \upsilon \text{ is in between } |s_0| \text{ and } 1\big\}$ lies inside $\bigcap_{c\in X} U^{\bt - c}_{1 ,\epsilon}$.
On the other hand, using Theorem \ref{thm:AAS-of-val-ring}(c), we see that $\{\zeta_{0_K, 1}^{\omega}\}_{\omega\geq \omega_0}$ converges to $\zeta_{0_F,1}$ when $\omega\to \infty$, and this produces a path joining $\zeta_{0_K, 1}^{\omega_0}$ to $\zeta_{0_F,1}$. 
Moreover, the equalities $\zeta_{0_K, 1}(\bt -c) = 1$ ($c\in X$) gives $\zeta_{0_K, 1}^{\omega}\in \bigcap_{c\in X} U^{\bt - c}_{1 ,\epsilon}$ ($\omega\geq \omega_0$).
Consequently, we obtain a path in $\bigcap_{c\in X} U^{\bt - c}_{1 ,\epsilon}\setminus \BA_K^1\times [1,\kappa]$ joining $\lambda_0^{\omega_0}$ to $\zeta_{0_F,1}$.

\smallskip\noindent
(c) Clearly, the countability of $F$ and the first countability of $\BA_K^1$ will imply the first countability of $\BA_R^1$. 
Conversely, suppose that $\BA_R^1$ is first countable.
Then $\BA_K^1 \times [1,\infty)$ is first countable and so is $\BA_K^1$.
Moreover, there is a countable subcollection
$$\left\{ \bigcap_{c\in X_k} U^{\bt - c}_{1 ,\frac{1}{n_k}}\setminus \BA_K^1\times [1,{m_k}]: k\in \BN\right\}$$
that form a neighborhood basis for $\zeta_{0_F,1}$ in $\BA_R^1$.
Let $C:= \bigcup_{k\in \BN} X_k$.
We know that whenever a net $\{(\lambda_i, \omega_i)\}_{i\in \KI}$ in $\BA_K^1\times [1,\infty)$ satisfying $\omega_i \to \infty$ as well as $|\bt - b|_{\lambda_i}^{\omega_i} \to 1$ ($b\in C$), we have $\lambda_i^{\omega_i} \to \zeta_{0_F,1}$.
Assume on the contrary that $\ti F$ is uncountable, then there exists $c_0\in \ti F\setminus C$ (which implies $|c_0-b| = 1$ for all $b\in C$).
If we set $\omega_n := n$ and $\lambda_n := \zeta_{c_0,0}$ ($n\in \BN$), then $|\bt - b|_{\lambda_n}^{\omega_n} = 1$ ($n\in \BN$), for each $b\in C$, but $|\bt - c_0|_{\lambda_n}^{\omega_n} = 0$ ($n\in \BN$), which contradicts $\lambda_n^{\omega_n} \to \zeta_{0_F,1}$.
\end{proof}

\begin{rem}\label{rem:open-connected-subset}
Unlike $\BA_K^1$, the topological space $\BA_R^1$ is far from having a $\BR$-tree structure. 
Actually, for any connected open subset $V\subseteq \BA_R^1$ and any $\lambda_1, \lambda_2\in V$, there exist infinitely many paths in $V$ joining $\lambda_1$ and $\lambda_2$. 

In fact, consider $i\in \{1,2\}$. 
Let $U_i\subseteq V$ be a path connected open neighborhood of $\lambda_i$ (see Theorem \ref{thm:connect}(b)) and fix any $(\mu_i, \kappa_i)\in \BA_K^1\times (1, \infty)\cap U_i$ (see Theorem \ref{thm:AAS-of-val-ring}). 
Thus, $\lambda_i$ is joined to $(\mu_i, \kappa_i)$ through a path inside $U_i$. 
Choose a connected open neighborhood $W_i$ of $\mu_i$ in $\BA_K^1$ as well as a number $\epsilon \in (0,1-\kappa_i)$ such that $W_i\times (\kappa_i - \epsilon, \kappa_i + \epsilon) \subseteq U_i$. 
Pick any $\nu_i\in W_i$. 
There exist infinitely many paths in $W_i\times (\kappa_i - \epsilon, \kappa_i + \epsilon)$ joining $(\mu_i, \kappa_i)$ to $(\nu_i, \kappa_i)$. 
As $V$ is path connected, there exists a path in $V$ joining $(\nu_1,\kappa_1)$ to $(\nu_2, \kappa_2)$. 
In this way, we obtain infinitely many paths joining $\lambda_1$ and $\lambda_2$. 
\end{rem}

Similar to Theorem \ref{thm:connect}(c), there is also a description of the second countability of $\BA_R^1$.
In fact, one has the following more general result.
This result could be a known fact, but since we do not find an explicit reference for it, we present its simple argument here.

\begin{prop}\label{prop:2nd-count}
If $(S, \|\cdot\|)$ is a commutative unital Banach ring with multiplicative norm, the following statements are equivalent.
\begin{enumerate}[S1).]
\item $\BA_S^1$ is second countable.
\item $S$ is separable as a metric space.
\item $\KM(S\{n^{-1}\bt\})$ is metrizable for every $n\in \BN$.
\end{enumerate}
\end{prop}
\begin{proof}
$(S1) \Rightarrow (S2)$.
For any $a\in S$, we defined $\zeta_a\in \BA_S^1$ by $\zeta_a(\bp):= \|\bp(a)\|$ ($\bp\in S[\bt]$).
It is easy to see that $a\mapsto \zeta_a$ is a homeomorphism from $S$ onto its image in $\BA_S^1$.
Thus, $S$ is also second countable and hence is separable.

\smallskip\noindent
$(S2) \Rightarrow (S3)$.
Suppose that $S$ contains a countable dense subset $S_0$.
We denote by $\CF_0$ the collection of non-empty finite subsets of $S[\bt]$ consisting of polynomials with coefficients in $S_0$.
For a fixed $n\in \BN$, since $|\bt|_{\mu}\leq n$ for any $\mu\in \KM(S\{n^{-1}\bt\})$, it is not hard to see that the countable family
$$\big\{E_{1/k}^X\cap \KM(S\{n^{-1}\bt\})\times \KM(S\{n^{-1}\bt\}): k\in \BN; X\in \CF_0\big\}$$
(see %
Relation \eqref{eqt:def-entour} for the meaning of $E_{1/k}^X$) forms a fundamental system of entourages for the Berkovich uniform structure on $\KM(S\{n^{-1}\bt\})$.
Consequently, the Berkovich uniform structure (and hence the topology defined by it) is pseudo-metrizable (see, e.g., Theorem 13 in chapter 6 of \cite{Ke}). 
However, since the topology on $\KM(S\{n^{-1}\bt\})$ is Hausdorff, we conclude that this topology is indeed metrizable.

\smallskip\noindent
$(S3) \Rightarrow (S1)$.
Since $\KM(S\{n^{-1}\bt\})$ is a compact metric space, it is second countable, and  so, %
the open subset $\BB^S_n \subseteq \KM(S\{n^{-1}\bt\})$ (see \eqref{eqt:def-open-n-ball}) is also second countable.
Consequently, $\BA_S^1$ is second countable.
\end{proof}

Clearly, the Banach ring $R$ is separable (equivalently, second countable) if and only if $K$ is separable as a metric space. 
Note also that if a complete valued field is algebraically closed and spherically complete, then it is not separable (see e.g. \cite[Remark 1.4]{HLR}). 
However, the completion of the algebraic closure of a separable complete valued field is again separable (see e.g. \cite[Remark 1.3]{HLR}).

\begin{eg}
Let $\kk$ be a field equipped with the trivial norm.
By Proposition \ref{prop:1-D-aff-anal-sp-tri-val}, one can define a metric on $\BA_\kk^1$ through the ``geodesic distance'' $d_G$ of two given points.
If $\kk$ is uncountable, then Proposition \ref{prop:2nd-count} tells us that $\BA_\kk^1$ is not metrizable, and hence the topology induced by $d_G$ is different from the pointwise convergence topology on $\BA_\kk^1$.

On the other hand, suppose that $\kk$ is at most countable and $\bp_1,\bp_2, \dots$ are all the elements in $\kk[\bt]_{\rm irr}$.
If we rescale the ``geodesic distance'' so that $d_G(\gamma_{\bp_k,0}, \gamma_{\bt,1}) \to 0$ when $k\to \infty$, then it is not hard to see that this metric defines the topology of pointwise convergence on $\BA_\kk^1$.
In this case, $\BA_\kk^1$ is the following closed subspace of $\BR^2$:
$$
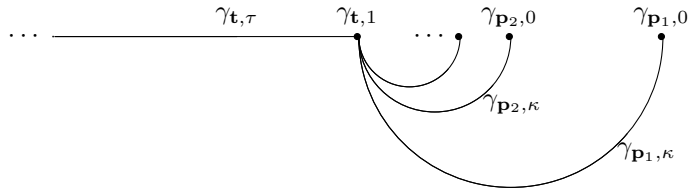
\begin{tikzpicture}[scale=2]

\draw (0,0) -- (1.2,0) node[above, black]{$\gamma_{\bt, \tau}$};

\draw (0,0) -- (2,0);

\draw[fill=black] (-0.01,0) circle(.002) node[left, black]{$\cdots$};

\draw[fill=black] (1.99,0) circle(.02) node[above, black]{$\gamma_{\bt, 1}$};

\draw (2,0) arc(180:320:1/3); 

\draw (2,0) arc(180:360:1/3);

\draw[fill=black] (2+2/3-0.01,0) circle(.02) node[left, black]{$\cdots$};

\draw (2,0) arc(180:320:1/2) node[below, black]{$\ \ \ \ \gamma_{\bp_2, \kappa}$};

\draw (2,0) arc(180:360:1/2);

\draw[fill=black] (2.99,0) circle(.02)  node[above, black]{$\gamma_{\bp_2, 0}$};

\draw (2,0) arc(180:360:1);

\draw (2,0) arc(180:320:1) node[below, black]{$\ \ \ \ \gamma_{\bp_1, \kappa}$};

\draw[fill=black] (3.99,0) circle(.02) node[above, black]{$\gamma_{\bp_1, 0}$}; 
\end{tikzpicture}$$
\end{eg}

In the following, we will have a look at the set of ``type I points'' of $\BA_R^1$, i.e. those multiplicative semi-norms $\lambda\in \BA_R^1$ with $\ker |\cdot|_\lambda \neq \{0_R\}$.

\begin{cor}\label{cor:top-on-type-I-points}
	Let $R$, $(K,|\cdot|)$, $F$ and $\ti F$ be as in Theorem \ref{thm:AAS-of-val-ring}.
	Denote
	$$\BA_{R,Z}^1:= \big\{\lambda\in \BA_R^1: \ker |\cdot|_\lambda\neq \{0_R\}\big\}.$$

	\smallskip\noindent
	(a) $\BA_{R,Z}^1 = \BA_F^1 \cup  K\times [1,\infty)$  (and hence, $\BA_{R,Z}^1$ is dense in $\BA_R^1$).
	
	\smallskip\noindent
	(b) Suppose that $\{(s_i, \omega_i)\}_{i\in \KI}$ is a net in $K\times [1, \infty)$.
	\begin{itemize}
		\item $\zeta_{s_i,0}^{\omega_i} \to \zeta_{Q(b), \tau_1}$ for some $b\in R$ and $\tau_1\in [0,1)$ if and only if $\omega_i\to \infty$ and $|s_i - b|^{\omega_i}\to \tau_1$;
		
		\item $\zeta_{s_i,0}^{\omega_i} \to \zeta_{0_F, \tau_2}$ for a number $\tau_2\in (1,\infty)$ if and only if $\omega_i\to \infty$ and  $|s_i|^{\omega_i}\to \tau_2$;

		\item $\zeta_{s_i,0}^{\omega_i} \to \zeta_{0_F,1}$ if and only if $\omega_i\to \infty$ and $|s_i - c|^{\omega_i}\to 1$, for any $c\in \ti F$.
	\end{itemize}

	\smallskip\noindent
	(c) If $\{(s_i, \omega_i)\}_{i\in \KI}$ is a net in $K\times [1, \infty)$ such that  $\zeta_{s_i,0}^{\omega_i} \to \zeta_{Q(b), \tau_1}$ for some $b\in R$ and $\tau_1\in [0,1)$, then $s_i$ eventually belongs to the ``open ball'' of $K$ of radius $1$ and center $b$.
\end{cor}
\begin{proof}
Since part (b) follows from Theorem  \ref{thm:AAS-of-val-ring} and part (c) follows directly from part (b), we will only establish part (a).
In fact, it is clear that  $\BA_F^1 \cup  K\times [1,\infty)\subseteq \BA_{R, Z}^1$.
Consider any element $\lambda\in \BA_{R, Z}^1$.
If $\ker |\cdot|_\lambda \cap R \neq \{0_R\}$, the argument of Proposition \ref{prop:ele-AAS-val-ring} tells us that $\lambda\in \BA_F^1$.
Suppose that $\ker |\cdot|_\lambda \cap R = \{0_R\}$.
As in the proof of Proposition \ref{prop:ele-AAS-val-ring}, there is a unique positive number $\omega\in [1,\infty)$ such that $\lambda$ extends to an element $\bar \lambda\in \BA_{K^\omega}^1$.
Since $\ker |\cdot|_{\bar \lambda}$ (which contains $\ker |\cdot|_{\lambda}$) is a non-zero prime ideal of $K[\bt]$ and $K$ is algebraically closed, one can find a (unique) element $s\in K$ with $\ker |\cdot|_{\bar \lambda} = (\bt -s)\cdot K[\bt]$ and it is not hard to check that $\lambda = \zeta_{s,0}^\omega$.
\end{proof}

In the following, we will use Theorem \ref{thm:AAS-of-val-ring} to obtain the  Berkovich spectra of Banach group rings of finite cyclic groups over $R$.
Let $G$ be a finite abelian group and $(S, \|\cdot\|)$ be a commutative unital Banach ring.
We denote by $S[G]$ the group ring of $G$ over $S$, and endowed it with the norm $\|\sum_{g\in G} a_g g\| := \max_{g\in G} \|a_g\|$.
Clearly, $S[G]$ is a commutative unital Banach ring.

\begin{cor}\label{cor:spec-cyclic-gp}
Let $R$, $F$, $Q$ and $(K, |\cdot|)$ be as in Theorem \ref{thm:AAS-of-val-ring}.
Denote by $|\cdot|_0$ the trivial norm on $F$. 
Let $G$ be a cyclic group of order $M$ with $u$ being a generator of $G$.
Suppose $b_1,\dots,b_n$ are all the distinct $M$-th roots of unity in $K$.

\smallskip\noindent
(a) If we set $\left|\sum_{l=0}^{M-1} a_l u^l\right|_{\alpha_{b_k}^{\omega}} := \left|\sum_{l=0}^{M-1} a_l b_k^l\right|^\omega$
and
$\left|\sum_{l=0}^{M-1} a_l u^l\right|_{\beta_{Q(b_k)}} := \left|\sum_{l=0}^{M-1} Q(a_l b_k^l)\right|_0$, then $\KM(R[G]) = \{\alpha_{b_k}^{\omega}: \omega\in [1,\infty); k=1,\dots,n\}\cup \{\beta_{Q(b_k)}: k=1,\dots,n\}$. 

\smallskip\noindent
(b) As a topological space, $\KM(R[G])$ consists of $n$ intervals of the form $[1,\infty]$ corresponding to the elements $b_1,\dots,b_n$ such that the ``$1$-ends'' of all these intervals are free, while the ``$\infty$-ends'' of the two intervals corresponding to $b_k$ and $b_l$ are identified with each other if $Q(b_k) = Q(b_l)$.
\end{cor}
\begin{proof}
(a) There is a contractive and surjective ring homomorphism
\begin{equation}\label{eqt:defn-q-G}
q_G:R\{ \bt \} \to R[G]
\end{equation}
sending $\bt$ to $u$, and it is not hard to check that $\ker q_G = (\bt^M -1)\cdot R\{ \bt \}$.
Hence, one may regard $\KM(R[G])$ as a topological subspace of $\KM(R\{\bt\})$ through $q_G$ in the following way:
$$\KM(R[G]) = \{\lambda\in \BA_{R}^1: |\bt|_\lambda \leq 1; |\bt^M -1|_\lambda = 0\} \subseteq \BA_{R,Z}^1.$$
It is obvious that $\alpha_{b_k}^\omega$ and $\beta_{Q(b_k)}$ are well-defined elements in $\KM(R[G])$, and they can be identified, respectively, with the elements $\zeta_{b_k,0}^\omega$ and $\zeta_{Q(b_k), 0}$ in $\BA_{R,Z}^1$.

On the other hand, let us pick an arbitrary element $\lambda\in \KM(R[G])$.
By Corollary \ref{cor:top-on-type-I-points}(a), either $\lambda = \zeta_{s,0}^\omega$ for a unique  $(s,\omega)\in K\times [1,\infty)$ or $\lambda\in \BA_{F}^1$.
In the first case, the condition $|\bt^M -1|_{\zeta_{s,0}^\omega} = 0$ will force $s = b_k$ for some $k\in \{1,\dots,n\}$, which means that $\lambda = \alpha_{b_k}^\omega$.
In the second case, there exist $x\in F$ and $\tau\in \RP$ satisfying $\lambda = \zeta_{x, \tau}$, and the condition $|\bt^M -1|_{\zeta_{x, \tau}} = 0$ tells us that $\tau = 0$ and $x^M = 1$.
Since $Q(b_1), \dots, Q(b_n)$ are all the roots of $\bt^M -1$ in $F$, we conclude that $\lambda = \beta_{Q(b_k)}$ for some $k=1,\dots,n$.

\smallskip\noindent
(b) By Corollary \ref{cor:top-on-type-I-points}(a), it is not hard to see that the subset %
$\{\zeta_{b_k,0}^\omega: \omega\in [1,\infty); k=1,\dots,n\}$ of $\BA_{R,Z}^1$ are $n$ disjoint intervals of the form $[1,\infty)$.
Assume that $(s_i, \omega_i)\in \{b_1,\dots,b_n\}\times [1,\infty)$ ($i\in \KI$) such that $\{\zeta_{s_i,0}^{\omega_i}\}_{i\in \KI}$ converges to $\zeta_{Q(b_k),0}$ for some $k\in \{1,\dots, n\}$.
Then Corollary \ref{cor:top-on-type-I-points}(c)
tells us that $|s_i - b_k| < 1$ eventually.
In other words, $Q(s_i) = Q(b_k)$ eventually.
Conversely, it follows from Corollary \ref{cor:top-on-type-I-points}(b) that the conditions $Q(s_i) = Q(b_k)$ for large $i$ and $\omega_i\to \infty$ will imply $\zeta_{s_i,0}^{\omega_i}\to \zeta_{ Q(b_k), 0}$.
This completes the proof.
\end{proof}

In the case when the field $K$ is not algebraically closed, one may use the above as well as \cite[Corollary 1.3.6]{Berk90} to describe $\KM(R[G])$.
In the following, we will consider the case when $R$ is the ring $\BZ_p$ of $p$-adic integers and $G$ is a cyclic $p$-group (for a fixed prime number $p$).
Let us start with the following possibly known lemma.

\begin{lem}\label{lem:irred-poly}
For any $l\in \BZ_+$, the polynomial $\bq_{l+1}:=\bt^{p^l(p-1)} + \bt^{p^l(p-2)} + \cdots + \bt^{p^l} + 1$ is irreducible in $\BZ_p[\bt]$.
\end{lem}
\begin{proof}
We first consider the case when $l =0$.
The equalities
$$\big((\bt +1) -1\big)\cdot \sum_{i=1}^p (\bt + 1)^{p-i} = (\bt + 1)^p - 1 = \bt \cdot \sum_{i=1}^p\binom{p}{i-1}\bt^{p-i},$$
gives $\sum_{i=1}^p (\bt + 1)^{p-i} = \sum_{i=1}^p\binom{p}{i-1}\bt^{p-i}$.
Thus, it follows from the Eisenstein's criterion that the polynomial $\sum_{i=1}^p (\bt + 1)^{p-i}$ is irreducible in $\BZ_p[\bt]$, and hence so is $\sum_{i=1}^p \bt^{p-i}$.

In the case of $l \geq 1$, we set $\bs:= (\bt + 1)^{p^l} - 1$.
Since
$\sum_{i=1}^p (\bs+1)^{p-i} = \sum_{i=1}^p\binom{p}{i-1}\bs^{p-i}$, we have
$$\sum_{i=1}^p (\bt+1)^{p^l(p-i)} = \sum_{i=1}^p\binom{p}{i-1}\big((\bt+1)^{p^l} - 1\big)^{p-i}.$$
From the left hand side, we see that the coefficient of $\bt^{p^l(p-1)}$ is $1$ and  the constant coefficient is $p$. 
From the right hand side, we know that all the other coefficients of $\bt^k$ are divisible by $p$.
Thus, the Eisenstein's criterion tells us that $\sum_{i=1}^p (\bt+1)^{p^l(p-i)}$ is an irreducible polynomial in $\BZ_p[\bt]$ and hence so is $\sum_{i=1}^p \bt^{p^l(p-i)}$.
\end{proof}

\begin{eg}\label{eg:order=p}
Let $G$ be the cyclic group of order $p^N$ for a positive integer $N$.
We set $\BF_p := \BZ/p\BZ$ and consider $q_G:\BZ_p\{ \bt\} \to \BZ_p[G]$ to be the canonical quotient map.
As in the argument of Corollary \ref{cor:spec-cyclic-gp}, we may identify, through $q_G$: %
$$\KM(\BZ_p[G]) = \Big\{\lambda\in Q^\BA(\BA_{\BF_p}^1) \cup \bigcup_{\omega\in [1,\infty)} J_\omega^\BA(\BA_{\BQ_p^\omega}^1): |\bt|_\lambda \leq 1; |\bt^{p^N} -1|_\lambda = 0\Big\}$$
(see %
Proposition \ref{prop:ele-AAS-val-ring}).
By Lemma \ref{lem:irred-poly}, the prime factorization of $\bt^{p^N} -1$ in $\BZ_p[\bt]$ is
\begin{equation*}
\bt^{p^N} -1 = \bq_{0}\cdot \bq_1\cdot \bq_2\cdots \bq_{N},
\end{equation*}
where $\bq_{0}:= \bt -1$.
On the other hand, the following is the prime factorization of $\bt^{p^N} -1$ in $\BF_p[\bt]$:
\begin{equation}\label{eqt:irrd-decomp-in-Fp}
\bt^{p^N} -1 = (\bt -1)^{p^N}.
\end{equation}

Let $\BC_p$ be the completion of the algebraic closure of $\BQ_p$, let $R_p$ be the ring of integers of $\BC_p$ and let $F_p$ be the residue field of $R_p$.
For $1\leq k \leq N$, we consider $\{r_{k,1},\cdots,r_{k, n_k}\}$ to be  the set of all distinct roots of $\bq_k$ in $\BC_p$.
It follows from \eqref{eqt:irrd-decomp-in-Fp} that $\bt^{p^N} - 1 = (\bt - 1)^{p^N}$ in $F_p[\bt]$.
Hence, $Q(r_{k,i}) = 1$ for all possible $k$ and $i$.
Therefore, Corollary \ref{cor:spec-cyclic-gp}(a) tells us that
$$\KM(R_p[G]) = \{\alpha_{r_{k,i}}^\omega: \omega\in [1,\infty); 1\leq k\leq N; 1\leq i\leq n_k\}\cup \{\alpha_1^\omega:\omega\in [1,\infty)\} \cup \{\beta_1\}.$$ %
As in Corollary \ref{cor:spec-cyclic-gp}(b), the topological space $\KM(R_p[G])$ consists of $1 + \sum_{k=1}^{N} n_k$ intervals of the form %
$[1,\infty]$ with all the ``$1$-ends'' being free but with all the ``$\infty$-ends'' being identified with one point, namely, $\beta_1$.

Again, the prime factorization as in \eqref{eqt:irrd-decomp-in-Fp} ensures that %
$\KM(\BZ_p[G])\cap Q^\BA\big(\BA_{\BF_p}^1\big) = \{Q^\BA(\zeta_{1_{\BF_p}, 0})\}$, and the semi-norm %
$\bar \beta_1:= Q^\BA(\zeta_{1_{\BF_p}, 0})$ coincides with the one induced by $\beta_1\in \KM(R_p[G])$ through restriction.
On the other hand, as in \cite[Corollary 1.3.6]{Berk90}, any element $\lambda\in \KM(\BZ_p[G])\cap J_\omega^\BA(\BA_{\BQ_p^\omega}^1)$ can be extended to an element $\bar{\lambda}\in J_\omega^\BA(\BA_{\BC_p^\omega}^1)$.
It follows from $|\bt|_{\bar\lambda} \leq 1$ and $|\bt^{p^N} -1|_{\bar\lambda} = 0$ that $\bar{\lambda}$ is either $\alpha_1^\omega$ or $\alpha_{r_{k,i}}^\omega$ for suitable $k$ and $i$.

Let us set $\bar \alpha_{0}^\omega$ to be the element in $\KM(\BZ_p[G])$ induced by $\alpha_1^\omega$.
On the other hand, for a fixed $k\in \{1,2,\dots,N\}$, by considering an automorphism in the Galois group of the splitting field of the irreducible polynomial $\bq_k$ over $\BQ_p$, we know that $\alpha_{r_{k,i}}^{\omega}$ and $\alpha_{r_{k,j}}^{\omega}$ restrict to the same element in $\KM(\BZ_p[G])$ for any $i,j\in \{1,\dots,n_k\}$.
We denote the resulting element by $\bar \alpha_{k}^\omega$.
As $\bar \alpha_{k}^\omega$ comes from different irreducible factors of $\bt^{p^N} - 1$ for different $k$, they are all distinct.

Consequently, as a quotient space of $\KM(R_p[G])$, the topological space $\KM(\BZ_p[G])$ is the following subspace of $\BR^2$:
$$
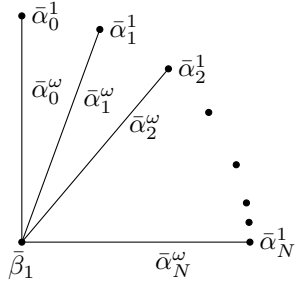
\begin{tikzpicture}[scale=2]

\draw (0,0) -- (0,1.5);

\draw[fill=black] (0,0) circle(.02) node[below, black]{$\bar \beta_1$};

\draw[fill=black] (0,1) node[right, black]{$\bar \alpha_{0}^{\omega}$};

\draw[fill=black] (0,1.5) circle(.02) node[right, black]{$\bar \alpha_{0}^{1}$};

\draw (0,0) -- (0.513, 1.410);

\draw[fill=black] (0.513, 1.410) circle(.02) node[right, black]{$\bar \alpha_{{1}}^{1}$};

\draw[fill=black] (0.342, 0.940) node[right, black]{$\bar \alpha_{{1}}^{\omega}$};

\draw (0,0) -- (0.964 , 1.149);

\draw[fill=black] (0.964 , 1.149) circle(.02) node[right, black]{$\bar \alpha_{{2}}^{1}$};

\draw[fill=black] (0.643 , 0.766) node[right, black]{$\bar \alpha_{{2}}^{\omega}$};

\draw[fill=black] (1.229 , 0.860) circle(.02);

\draw[fill=black] (1.410 , 0.513) circle(.02);

\draw[fill=black] (1.477 , 0.260) circle(.02);

\draw[fill=black] (1.494 , 0.131) circle(.02);

\draw (0,0) -- (1.5,0);

\draw[fill=black] (1.5,0) circle(.02) node[right, black]{$\bar \alpha_{{N}}^{1}$};

\draw[fill=black] (1,0) node[below, black]{$\bar \alpha_{{N}}^{\omega}$};

\end{tikzpicture}$$
\end{eg}

\section*{Acknowledgement}

This work is supported by the National Natural Science Foundation of China (11471168).


\begin{thebibliography}{9}

\bibitem{AM}
M.F. Atiyah and I.G. Macdonald, \emph{Introduction to Commutative algebra}, Addison-Wesley Publ. Co. (1969). 

\bibitem{BR}
M. Baker and R. Rumely, \emph{Potential Theory and Dynamics on the Berkovich Projective Line}, Math. Surveys and Mono. \textbf{159}, Amer. Math. Soc. (2010).

\bibitem{Berk90}
V.G. Berkovich, \emph{Spectral theory and Analytic Geometry over non-Archimedean fields}, Math.\ Surveys and Mono.\ \textbf{33}, Amer.\ Math.\ Soc.\ (1990).

\bibitem{Bour}
N. Bourbaki, \emph{Commutative Algebra Chapters 1-7}, Springer-Varlag (1989).

\bibitem{HLR}
E. Hrushovski, F. Loeser and B. Poonen, Berkovich spaces embed in Euclidean spaces, L'Enseign. Math. \textbf{60} (2014), 273-292. 

\bibitem{Ke}
J. L. Kelley, \emph{General Topology}. GTM \textbf{27}, Springer (1975).

\bibitem{LN-Ber}
C.W. Leung and C.K. Ng, Berkovich spectrum for elements in Banach rings, preprint (arXiv:1410.5893).

\bibitem{Poin}
J. Poineau, La droite de Berkovich sur Z, Ast\'{e}risque, \textbf{334} (2010). 

\bibitem{RTW}
B. Remy, A. Thuillier, and A. Werner, Bruhat-Tits Buildings and Analytic Geometry, in A. Ducros et al. (eds), \emph{Berkovich Spaces and Applications}, Lect. Note Math. \textbf{2119}, Springer (2015).

\bibitem{W}
A. Werner, Non-Archimedean Analytic Spaces, Jahresber Dtsch Math-Ver \textbf{115}, 3-20, (2013).

\end{thebibliography}
\end{document}